\documentclass[11pt]{article}

\usepackage{amsmath,amssymb,amsthm,mathtools}
\usepackage{hyperref}
\usepackage{enumitem}
\usepackage{geometry}
\newtheorem{theorem}{Theorem}
\newtheorem{proposition}{Proposition}
\newtheorem{lemma}{Lemma}
\newtheorem{corollary}{Corollary}
\newtheorem{example}{Example}
\newtheorem{remark}{Remark}

\title{Beyond Bass Collapse: New Irregular Edge-Space\\ Invariants in Ihara Theory}

\author{Hartosh Singh Bal}
\date{}

\begin{document}
\maketitle

\begin{abstract}
Let \(G\) be a finite simple graph and let \(T\) be its Hashimoto operator on the directed-edge space. We show that edge reversal induces a canonical symmetric/antisymmetric splitting under which \(T\) acquires an explicit \(2\times 2\) block form. The diagonal blocks are \(\tfrac12 L(G)\) and \(-\tfrac12 A(G)\), where \(L(G)\) is the line-graph adjacency and \(A(G)\) is the antisymmetric line-graph adjacency, while the off-diagonal block is the mixed incidence product \(M=|D|^\top D\). This identifies the ordinary and antisymmetric line-graph sectors as the two canonical diagonal sectors of Hashimoto theory and isolates a mixed sector linking them.

A Schur-complement argument then gives a factorization
\[
\det(I-wT)=\det\!\bigl(I-\tfrac w2 L(G)\bigr)\,C_G(w),
\]
where \(C_G(w)\) is an explicit correction determinant built from the antisymmetric and mixed sectors. We show that the trivial roots \(w=\pm1\) localize on canonical edge subspaces, and that for line-graph-cospectral pairs all remaining Ihara separation is forced into the correction sector.

Although the raw mixed block \(M\) depends on edge orientation, its natural gauge-invariant shadows, including \(MM^\top\), \(M^\top M\), and \(M^\top L^kM\), define a canonical matrix package attached to the graph. In the regular case these collapse to adjacency-side data, but in the irregular case they need not. As an application, we exhibit irregular non-isomorphic graphs that are adjacency-cospectral and line-graph-cospectral yet are separated by the correction sector, and we find further examples where the gauge-invariant mixed shadows separate even when the scalar Ihara polynomial does not. This isolates new irregular edge-space invariants in Hashimoto--Ihara theory.
\end{abstract}

\medskip
\noindent\textbf{2020 Mathematics Subject Classification.}
05C50, 05C22, 11M41.

\noindent\textbf{Keywords.}
Ihara zeta function, Hashimoto operator, non-backtracking matrix, line graph, signed graph, antisymmetric line graph, graph spectra, cospectral graphs.

\bigskip

\section{Introduction}

The Ihara zeta function of a finite graph $G$ is a generating function for primitive non-backtracking cycles, and admits the determinant representation \cite{Ihara1966,Hashimoto1989}
\[
\zeta_G(w)^{-1}=\det(I-wT),
\]
where $T$ is the Hashimoto non-backtracking operator on the $2m$ directed edges \cite{Hashimoto1989}. Bass' formula rewrites the same determinant in vertex space \cite{Bass1992}:
\[
\det(I-wT)=(1-w^2)^{m-n}\det(I-wA_G+w^2(D_G-I)),
\]
where $n=|V|$, $m=|E|$, $A_G$ is the adjacency matrix, and $D_G$ is the diagonal degree matrix. In particular, in the regular regime the scalar Ihara theory collapses to ordinary adjacency data.

The Ihara zeta function, its determinant formulas, and their relationship to graph coverings form a classical part of algebraic graph theory; for general background from this perspective see, for example, Terras \cite{TerrasBook}, Stark--Terras \cite{StarkTerras1996}, and Kotani--Sunada \cite{KotaniSunada2000}.

This scalar collapse leaves open a natural edge-space question: what structure, if any, survives on the directed-edge state space before one passes to the final determinant? Put differently, once the non-backtracking dynamics is viewed on oriented edges rather than in vertex space, where does genuinely new graph-theoretic information live, and how much of it is lost in the Bass reduction?

In parallel with this question, one may study canonical edge-based graph structures attached to \(G\). In \cite{BalHL2} we introduced a canonical graph \(\mathrm{HL}'_2(G)\) on the \(2m\) directed edges recording shared-head and shared-tail adjacency. In \cite{BalALG} we identified the antisymmetric \((-1)\)-sector of the edge-reversal involution as a canonical signed refinement of the line graph, denoted \(\mathcal A(G)\). The signed adjacency \(D^{\mathsf T}D-2I\) itself belongs to the classical signed-line-graph literature; see, for instance, Zaslavsky’s survey and bibliography \cite{ZaslavskySGG,ZaslavskyBiblio}. What is new in the present setting is that the relevant switching class is canonically attached to the underlying graph and is forced by the edge-reversal symmetry of the directed-edge space. These constructions suggest that the non-backtracking state space carries more internal structure than is visible from the scalar Ihara determinant alone.

The present paper makes this precise. The key observation is that the directed-edge space on which \(T\) acts is exactly the space underlying \(\mathrm{HL}'_2(G)\), and that edge reversal furnishes a canonical \(\mathbb Z_2\)-symmetry of this space. Resolving that symmetry yields four layers of structure:

\begin{itemize}[leftmargin=2.2em]
\item a canonical sector decomposition of the Hashimoto operator in which the two diagonal sectors are exactly \(\frac12 L(G)\) and \(-\frac12\mathcal A(G)\),
\item a Schur-complement factorization of \(\det(I-wT)\) into an ordinary line-graph factor and an explicit correction factor,
\item a mixed transport block \(|D|^\top D\) coupling the two sectors, together with its gauge-invariant mixed shadows \(MM^\top\), \(M^\top M\), and \(M^\top L^kM\),
\item structural consequences of this factorization, including localization of the trivial roots on canonical edge subspaces and a resolution principle showing that, for line-graph-cospectral pairs, all remaining Ihara separation is forced into the correction sector.
\end{itemize}

The resulting picture is a hierarchy of edge-space information beyond the classical Bass collapse. The raw mixed block $M:=|D|^\top D$ is a gauge-dependent transport representative; its natural graph invariants are the gauge-invariant mixed shadows built from it; and the scalar correction determinant is a further shadow of this matrix package. In the regular regime the primary mixed shadows collapse back to adjacency-side data, but in the irregular regime this collapse fails. The main theme of the paper is that this failure produces genuinely new irregular edge-space invariants in Ihara theory.

Our first application is a pair of non-isomorphic graphs that are simultaneously adjacency-cospectral and line-graph-cospectral but have distinct Ihara zeta functions, with first divergence forced entirely into the correction sector. We then show that this is not an isolated phenomenon: small-census irregular examples already exhibit gauge-invariant mixed-shadow separation, including a control pair in which the mixed-shadow package separates while the scalar Hashimoto determinant does not. Thus the paper identifies a new layer of irregular graph invariants beyond the ordinary line graph, beyond the antisymmetric line graph, and in part beyond the scalar Ihara polynomial itself.

\section{Directed edges, edge-reversal, and the symmetric lift}

\subsection{Directed-edge space and Hashimoto operator}

Let $G=(V,E)$ be a finite simple graph. Let $\vec E$ denote the set of directed edges:
for each $\{u,v\}\in E$ we have two directed edges $(u,v)$ and $(v,u)$. Thus $|\vec E|=2m$.

Index $\mathbb R^{\vec E}$ by directed edges. The Hashimoto operator $T$ is defined by
\[
T_{(u,v),(x,y)}=
\begin{cases}
1 &\text{if } v=x \text{ and } y\neq u,\\
0 &\text{otherwise.}
\end{cases}
\]
Equivalently, $T$ encodes one-step non-backtracking moves $(u,v)\mapsto(v,y)$ with $y\neq u$.

\subsection{Edge reversal}

Define the involution $\iota:\vec E\to\vec E$ by $\iota(u,v)=(v,u)$ and let $P$ be its permutation matrix:
\[
P e_{(u,v)} = e_{(v,u)},\qquad P^2=I.
\]
Hence $\mathbb R^{\vec E}$ decomposes canonically as
\[
\mathbb R^{\vec E}=E_+\oplus E_-,
\qquad
E_\pm = \{x\in\mathbb R^{\vec E}:Px=\pm x\}.
\]

\subsection{The symmetric lift $\mathrm{HL}'_2(G)$ and the identity $A=PT+TP$}

Following \cite{BalHL2}, define $\mathrm{HL}'_2(G)$ as the graph on vertex set $\vec E$ in which two directed edges are adjacent iff they share a tail or share a head, and are distinct. Let $A$ denote its adjacency matrix.

\begin{proposition}[State-space alignment]\label{prop:APTTP}
With $A,T,P$ as above,
\[
A = PT + TP.
\]
Consequently, using $P^2=I$ and the identity $T = PT^\top P$ for undirected graphs, one has
\[
PA = T+T^\top.
\]
\end{proposition}

\begin{proof}
Fix directed edges $e=(u,v)$ and $f=(x,y)$.

First consider $(TP)_{e,f}=T_{e,f^{-1}}$. This equals $1$ iff
\[
\text{head}(e)=\text{tail}(f^{-1}) \iff v = y,
\]
and the non-backtracking condition for $T_{(u,v),(y,x)}$ is $x\neq u$, i.e. $f\neq e$. Thus $(TP)_{e,f}=1$ exactly when $e$ and $f$ share their head and are distinct, which is exactly one of the two adjacency conditions in $\mathrm{HL}'_2(G)$.

Next consider $(PT)_{e,f}=T_{e^{-1},f}=T_{(v,u),(x,y)}$. This equals $1$ iff
\[
\text{head}(e^{-1})=\text{tail}(f)\iff u=x,
\]
and the non-backtracking condition for $T_{(v,u),(u,y)}$ is $y\neq v$, i.e. $f\neq e^{-1}$. But since $\mathrm{HL}'_2(G)$ disallows equality of vertices and $f=e^{-1}$ would share both head and tail with $e$, the adjacency criterion is precisely: share tail and be distinct. Hence $(PT)_{e,f}=1$ exactly when $e$ and $f$ share their tail and are distinct.

Since adjacency in $\mathrm{HL}'_2(G)$ is the union of these disjoint conditions, we have $A=PT+TP$.

For the second identity, on undirected graphs one checks $T_{e,f}=T_{f^{-1},e^{-1}}$, which is the matrix relation $T=PT^\top P$. We record the standard symmetry $T = PT^\top P$ for undirected graphs: indeed,
$T_{e,f}=1$ iff $\mathrm{head}(e)=\mathrm{tail}(f)$ and $f\neq e^{-1}$, which is equivalent to
$T_{f^{-1},e^{-1}}=1$. In matrix form this reads $T = PT^\top P$. Multiply $A=PT+TP$ on the left by $P$:
\[
PA = P(PT+TP) = T + PTP = T + T^\top,
\]
using $PTP = (PT^\top P)^\top = T^\top$.  \end{proof}

\section{Incidence matrices and the line graph / antisymmetric line graph}

Fix a reference orientation of the undirected edges $E$ (one direction per edge). Write an undirected edge as $e=\{a,b\}$ and its chosen orientation as $a\to b$.

Define the oriented incidence matrix $D\in\mathbb R^{V\times E}$ by
\[
D_{v,e}=
\begin{cases}
+1 &\text{if } v \text{ is the head of } e,\\
-1 &\text{if } v \text{ is the tail of } e,\\
0 &\text{otherwise.}
\end{cases}
\]
Let $|D|$ denote the entrywise absolute value (unsigned incidence).

\begin{lemma}[Line graph and antisymmetric line graph from incidence]\label{lem:LSfrominc}
Let $I$ denote the $m\times m$ identity. Then
\[
L := |D|^\top|D| - 2I
\]
is the adjacency matrix of the line graph $L(G)$ (in the edge-indexing of $E$), and
\[
S := D^\top D - 2I
\]
is the signed adjacency matrix of the antisymmetric line graph $\mathcal A(G)$ in the same edge-indexing. In particular, $S$ is well-defined up to switching.
\end{lemma}

\begin{proof}
For distinct edges $e\neq f$, the entry $(|D|^\top|D|)_{e,f}$ counts the number of vertices incident to both $e$ and $f$. In a simple graph, two distinct edges share either $0$ or $1$ vertices, hence $(|D|^\top|D|)_{e,f}=1$ iff $e,f$ are adjacent in $L(G)$. Diagonal entries satisfy $(|D|^\top|D|)_{e,e}=2$, so subtracting $2I$ yields the line graph adjacency.

Similarly, for distinct edges $e\neq f$ sharing a vertex $v$, we have $D_{v,e},D_{v,f}\in\{\pm1\}$ and
\[
(D^\top D)_{e,f}=\sum_{v\in V} D_{v,e}D_{v,f}=D_{v,e}D_{v,f}\in\{\pm1\}.
\]
This sign equals $+1$ when $e$ and $f$ are both oriented into $v$ or both oriented out of $v$, and equals $-1$ when exactly one is oriented into $v$. Thus $D^\top D-2I$ is a signed adjacency on $E$ with underlying graph $L(G)$. Switching corresponds to reversing the reference orientation of an edge, which multiplies the corresponding column of $D$ by $-1$ and thus conjugates $S$ by a diagonal $\{\pm1\}$-matrix.  \end{proof}

\section{Sector decomposition of the Hashimoto operator}

We now prove the main structural theorem: in the $(\pm)$-eigenbasis of $P$, the Hashimoto operator has explicit blocks involving $L,S$, and the mixed incidence product $|D|^\top D$.

\subsection{The $(\pm)$-basis on directed edges}

Index undirected edges by $E$. For each undirected edge $e=\{a,b\}$, let the two directed edges be $e^+=(a,b)$ and $e^-=(b,a)$ (these depend on the chosen reference orientation; we fix one for the proof).

Define vectors in $\mathbb R^{\vec E}$:
\[
\phi_e^+ := \frac{1}{\sqrt2}\bigl( \mathbf e_{e^+}+\mathbf e_{e^-}\bigr),
\qquad
\phi_e^- := \frac{1}{\sqrt2}\bigl( \mathbf e_{e^+}-\mathbf e_{e^-}\bigr),
\]
where $\mathbf e_{(u,v)}$ is the standard basis vector of $\mathbb R^{\vec E}$.

Then $P\phi_e^\pm=\pm\phi_e^\pm$, so $\{\phi_e^+\}_{e\in E}$ is an orthonormal basis of $E_+$ and $\{\phi_e^-\}_{e\in E}$ is an orthonormal basis of $E_-$.

Let $Q$ denote the orthogonal change-of-basis matrix sending the ordered basis
\[
(\mathbf e_{e^+},\mathbf e_{e^-})_{e\in E}
\quad\mapsto\quad
(\phi_e^+,\phi_e^-)_{e\in E}.
\]
Then $Q^\top T Q$ is the block matrix of $T$ in the $E_+\oplus E_-$ decomposition.

\subsection{Entrywise computation of blocks}

\begin{theorem}[Sector decomposition]\label{thm:sector}
With $L,S$ as in Lemma~\ref{lem:LSfrominc}, one has
\[
Q^\top T Q
=
\begin{pmatrix}
\frac12 L & -\frac12\,|D|^\top D\\[4pt]
\frac12\,D^\top|D| & -\frac12 S
\end{pmatrix}.
\]
Equivalently, the restriction of $T$ to the symmetric sector is $\frac12 L$, the restriction to the antisymmetric sector is $-\frac12 S$, and the cross-coupling is the mixed incidence product.
\end{theorem}

\begin{proof}
Fix a reference orientation of each undirected edge and write the two directed realizations
of an undirected edge $e$ as $e^+$ (the reference direction) and $e^-=(e^+)^{-1}$.
Define the orthonormal $(\pm)$-basis vectors
\[
\phi_e^+ := \frac{1}{\sqrt2}\bigl(\mathbf e_{e^+}+\mathbf e_{e^-}\bigr),
\qquad
\phi_e^- := \frac{1}{\sqrt2}\bigl(\mathbf e_{e^+}-\mathbf e_{e^-}\bigr),
\]
so that $P\phi_e^\pm=\pm\phi_e^\pm$.

For $\alpha,\beta\in\{+,-\}$ define block entries
\[
T_{\alpha\beta}(e,f):=\langle \phi_e^\alpha,\,T\phi_f^\beta\rangle.
\]
Expanding $\phi_e^\alpha,\phi_f^\beta$ gives the universal identity
\begin{equation}\label{eq:block-expand}
T_{\alpha\beta}(e,f)
=
\frac12\Bigl(
T_{e^+,f^+}+\beta\,T_{e^+,f^-}+\alpha\,T_{e^-,f^+}+\alpha\beta\,T_{e^-,f^-}
\Bigr).
\end{equation}

\medskip
\noindent\textbf{Local structure: exactly one contributing directed pair.}
Assume $e\neq f$ share a vertex $w$ (otherwise all four Hashimoto entries in
\eqref{eq:block-expand} are $0$). Write $e^+=(a\to b)$ and $f^+=(c\to d)$.
Since $T_{(u,v),(x,y)}=1$ iff $v=x$ and $y\neq u$, and $e,f$ share only the single
vertex $w$, there is \emph{exactly one} choice among $(e^\pm,f^\pm)$ for which the head
of $f^\pm$ equals the tail of $e^\pm$ (hence $T=1$), and the non-backtracking
constraint is automatic because the other endpoints are distinct.

More precisely, depending on the position of $w$ with respect to the reference orientations,
exactly one of the following four situations occurs:

\[
\begin{array}{c|c|c|c|c}
\text{Case} & \text{Position of }w & \text{Nonzero entry} & D_{w,e} & D_{w,f}\\\hline
\mathrm{A} & \mathrm{head}(e^+)=\mathrm{tail}(f^+) & T_{e^+,f^+}=1 & +1 & -1\\
\mathrm{B} & \mathrm{tail}(e^+)=\mathrm{tail}(f^+) & T_{e^-,f^+}=1 & -1 & -1\\
\mathrm{C} & \mathrm{head}(e^+)=\mathrm{head}(f^+) & T_{e^+,f^-}=1 & +1 & +1\\
\mathrm{D} & \mathrm{tail}(e^+)=\mathrm{head}(f^+) & T_{e^-,f^-}=1 & -1 & +1
\end{array}
\]
(All other entries among $T_{e^\pm,f^\pm}$ vanish.)

\medskip
\noindent\textbf{Computation of the blocks.}
Using \eqref{eq:block-expand} and the fact that exactly one of the four terms is $1$:

\smallskip
\noindent\emph{(i) The symmetric block.}
For $\alpha=\beta=+$, \eqref{eq:block-expand} becomes
\[
T_{++}(e,f)=\tfrac12\bigl(T_{e^+,f^+}+T_{e^+,f^-}+T_{e^-,f^+}+T_{e^-,f^-}\bigr).
\]
Hence for adjacent distinct $e,f$ one has $T_{++}(e,f)=\tfrac12$, and for non-adjacent
edges $T_{++}(e,f)=0$. Since the diagonal is $0$, this is exactly $T_{++}=\tfrac12 L$.

\smallskip
\noindent\emph{(ii) The antisymmetric block.}
For $\alpha=\beta=-$,
\[
T_{--}(e,f)=\tfrac12\bigl(T_{e^+,f^+}-T_{e^+,f^-}-T_{e^-,f^+}+T_{e^-,f^-}\bigr).
\]
For example, in Case A the unique nonzero entry is $T_{e^+,f^+}=1$.
In \eqref{eq:block-expand} with $\alpha=\beta=-$, this entry appears with coefficient $+1$,
so $T_{--}(e,f)=\tfrac12$. Since the table gives $D_{w,e}=+1$ and $D_{w,f}=-1$ in Case A,
we have $-\tfrac12 D_{w,e}D_{w,f}=+\tfrac12$, as claimed. The remaining cases are analogous.
In Cases A--D, substituting the unique nonzero entry gives
\[
T_{--}(e,f)= -\tfrac12\,D_{w,e}D_{w,f}.
\]
Since $(D^\top D)_{e,f}=\sum_{v}D_{v,e}D_{v,f}=D_{w,e}D_{w,f}$ for adjacent $e\neq f$
(and $(D^\top D)_{e,e}=2$), it follows that
\[
T_{--} = -\tfrac12(D^\top D-2I)= -\tfrac12 S.
\]

\smallskip
\noindent\emph{(iii) The cross blocks.}
For $(\alpha,\beta)=(+,-)$,
\[
T_{+-}(e,f)=\tfrac12\bigl(T_{e^+,f^+}-T_{e^+,f^-}+T_{e^-,f^+}-T_{e^-,f^-}\bigr),
\]
and substituting the unique nonzero entry in Cases A--D yields
\[
T_{+-}(e,f)= -\tfrac12\,D_{w,f}.
\]
But for adjacent $e\neq f$,
\[
(|D|^\top D)_{e,f}=\sum_v |D|_{v,e}D_{v,f} = |D|_{w,e}D_{w,f}=D_{w,f},
\]
so $T_{+-}=-\tfrac12|D|^\top D$.

Similarly, for $(\alpha,\beta)=(-,+)$ one obtains
\[
T_{-+}(e,f)= \tfrac12\,D_{w,e}=\tfrac12(D^\top|D|)_{e,f},
\]
hence $T_{-+}=\tfrac12 D^\top|D|$.

Collecting (i)--(iii) gives the asserted block decomposition.
\end{proof}

\begin{remark}[What is canonical]
The sector decomposition is canonically attached to the edge-reversal involution $P$. In particular, the symmetric block is the ordinary line-graph operator, and the antisymmetric block is the antisymmetric line-graph operator, both determined intrinsically by the underlying undirected graph (the latter up to switching). The cross term depends on the choice of reference orientation, but only through the expected switching covariance; the correction determinant extracted below is therefore an invariant of the underlying graph rather than of the auxiliary orientation choice.
\end{remark}

\section{Schur-complement factorization of $\det(I-wT)$}

Theorem~\ref{thm:sector} is the structural core of the paper. Its first global consequence is that the Ihara determinant itself admits a sector factorization: one factor comes from the ordinary line-graph sector, while the second is an explicit signed correction term built from the antisymmetric line graph and the mixed incidence operator.

Write $T$ in blocks with respect to $E_+\oplus E_-$:
\[
T=\begin{pmatrix}T_{++} & T_{+-}\\ T_{-+} & T_{--}\end{pmatrix}
=
\begin{pmatrix}
\frac12 L & -\frac12|D|^\top D\\[2pt]
\frac12 D^\top|D| & -\frac12 S
\end{pmatrix}.
\]

\begin{lemma}[Block determinant / Schur complement]\label{lem:schur}
If $I-w\mathcal A$ is invertible for a block matrix
$\begin{pmatrix}\mathcal A & \mathcal B\\ \mathcal C & \mathcal D\end{pmatrix}$, then
\[
\det\!\begin{pmatrix}I-w\mathcal A & -w\mathcal B\\ -w\mathcal C & I-w\mathcal D\end{pmatrix}
=\det(I-w\mathcal A)\cdot \det\!\bigl(I-w\mathcal D-w^2\mathcal C(I-w\mathcal A)^{-1}\mathcal B\bigr).
\]
\end{lemma}

\begin{proof}
This is the standard block determinant identity:
\[
\begin{pmatrix}I-w\mathcal A & -w\mathcal B\\ -w\mathcal C & I-w\mathcal D\end{pmatrix}
=
\begin{pmatrix}I & 0\\ -w\mathcal C(I-w\mathcal A)^{-1} & I\end{pmatrix}
\begin{pmatrix}I-w\mathcal A & -w\mathcal B\\ 0 & I-w\mathcal D-w^2\mathcal C(I-w\mathcal A)^{-1}\mathcal B\end{pmatrix},
\]
and taking determinants yields the claim.
\end{proof}
The Schur-complement identity is valid whenever $I-\frac{w}{2}L$ is invertible, which holds for all sufficiently small $|w|$.
Multiplying the block determinant identity by $\det(I-\tfrac w2 L)$ eliminates the inverse and yields an identity between polynomials in $w$
that holds on a neighborhood of $0$; hence it holds identically for all $w$.
\begin{theorem}[Edge-sector factorization]\label{thm:factor}
Let $L,S,D,|D|$ be as above. Then
\[
\det(I-wT)=\det\!\bigl(I-\tfrac w2 L\bigr)\cdot C_G(w),
\]
where
\[
C_G(w)=\det\!\Bigl(I+\tfrac w2 S+\tfrac{w^2}{4}\,D^\top|D|\,(I-\tfrac w2 L)^{-1}|D|^\top D\Bigr).
\]
\end{theorem}

\begin{proof}
Apply Lemma~\ref{lem:schur} to the block matrix $T=\begin{pmatrix}T_{++} & T_{+-}\\ T_{-+} & T_{--}\end{pmatrix}$
from Theorem~\ref{thm:sector}, with
\[
T_{++}=\tfrac12 L,\quad T_{+-}=-\tfrac12|D|^\top D,\quad T_{-+}=\tfrac12 D^\top|D|,\quad T_{--}=-\tfrac12 S.
\]
Then
\[
\det(I-wT)=\det(I-\tfrac w2 L)\cdot \det\!\Bigl(I+w\tfrac12 S - w^2\cdot \tfrac12 D^\top|D|\,(I-\tfrac w2 L)^{-1}\tfrac12|D|^\top D\Bigr).
\]
The minus signs combine to give a plus:
\[
-w^2\cdot\Bigl(\tfrac12\Bigr)\Bigl(-\tfrac12\Bigr)=+\tfrac{w^2}{4}.
\]
Therefore the second factor is exactly
\[
\det\!\Bigl(I+\tfrac w2 S+\tfrac{w^2}{4}\,D^\top|D|\,(I-\tfrac w2 L)^{-1}|D|^\top D\Bigr).
\]
\end{proof}

\subsection{Spectral consequences of the sector decomposition}\label{subsec:spectral-consequences}

A first payoff of Theorem~\ref{thm:sector} is that the Hermitian and skew-Hermitian parts of the Hashimoto operator acquire a canonical sector interpretation. Because the cross-sector blocks satisfy
$T_{-+}=-T_{+-}^\top$, the symmetric and antisymmetric sectors decouple completely in the Hermitian part. Thus the decomposition is not only algebraic: it immediately yields spectral control on the location of the Hashimoto spectrum in the complex plane.

\begin{proposition}[Numerical-range bounds for the Hashimoto spectrum]\label{prop:numerical-range}
Let $T$ be the Hashimoto matrix of a finite simple graph $G$, and write $T$ in the
$P$-eigenbasis as in Theorem~\ref{thm:sector}:
\[
T \sim
\begin{pmatrix}
\frac12 L & T_\times\\[2pt]
- T_\times^\top & -\frac12 S
\end{pmatrix},
\qquad
T_\times=-\frac12|D|^\top D.
\]
Let
\[
H=\frac{T+T^\top}{2},\qquad K=\frac{T-T^\top}{2}.
\]
Then every eigenvalue $\lambda\in\operatorname{Spec}(T)$ satisfies:
\begin{enumerate}
\item[(i)] \emph{(real-part bounds)}
\[
-\frac12\,\rho(S)\ \le\ \Re(\lambda)\ \le\ \frac12\,\rho(L).
\]
\item[(ii)] \emph{(imaginary-part bounds)}
\[
|\Im(\lambda)|\ \le\ \|K\|\ =\ \|T_\times\|\ =\ \frac12\,\sigma_{\max}(|D|^\top D)
\ \le\ \frac12\sqrt{\rho(\Delta)\,\rho(Q)},
\]
where $\sigma_{\max}$ denotes the largest singular value, $\|\cdot\|$ the operator norm,
$\Delta:=DD^\top=D_G-A_G$ is the (vertex) Laplacian, and $Q:=|D||D|^\top=D_G+A_G$ is the signless Laplacian.
\end{enumerate}
\end{proposition}

\begin{proof}
Since $T_{-+}=-T_{+-}^\top$, the off-diagonal blocks cancel in $T+T^\top$, hence in the
$P$-eigenbasis one has
\[
H=\frac{T+T^\top}{2}\sim
\frac12
\begin{pmatrix}
L & 0\\
0 & -S
\end{pmatrix}.
\]
For any eigenpair $Tx=\lambda x$ with $x\neq0$,
\[
\Re(\lambda)=\frac{x^* Hx}{x^* x},
\]
so $\Re(\lambda)$ lies in the interval $[\lambda_{\min}(H),\lambda_{\max}(H)]$.
Because $\operatorname{Spec}(H)=\frac12\operatorname{Spec}(L)\cup\bigl(-\frac12\operatorname{Spec}(S)\bigr)$,
this gives (i), and in particular
\[
\lambda_{\max}(H)\le \frac12\rho(L),\qquad \lambda_{\min}(H)\ge -\frac12\rho(S).
\]

For (ii), note that in the same basis
\[
K=\frac{T-T^\top}{2}\sim
\begin{pmatrix}
0 & T_\times\\
- T_\times^\top & 0
\end{pmatrix}.
\]
Every eigenvalue of $T$ lies in the numerical range
\[
W(T)=\left\{\frac{x^*Tx}{x^*x}:x\in\mathbb{C}^{2m}\setminus\{0\}\right\}.
\]
Writing $T=H+K$ with $H^\top=H$ and $K^\top=-K$, we have
\[
\Im\!\left(\frac{x^*Tx}{x^*x}\right)=
\Im\!\left(\frac{x^*Kx}{x^*x}\right)
=
\Re\!\left(\frac{x^*(iK)x}{x^*x}\right),
\]
and $iK$ is Hermitian with $\|iK\|=\|K\|$. Hence
\[
\left|\Im\!\left(\frac{x^*Tx}{x^*x}\right)\right|
\le \|K\|
\quad\text{for all }x\neq0,
\]
so $|\Im(\lambda)|\le\|K\|$ for every $\lambda\in\operatorname{Spec}(T)$.

Finally, for $K\sim\begin{psmallmatrix}0&B\\-B^\top&0\end{psmallmatrix}$ the singular values of $K$
are exactly those of $B$ (each with multiplicity $2$), so $\|K\|=\|B\|$; applying this with
$B=T_\times$ yields $\|K\|=\|T_\times\|$, and $\|T_\times\|=\frac12\sigma_{\max}(|D|^\top D)$ from
$T_\times=-\frac12|D|^\top D$.

For the final inequality, note that
\[
\sigma_{\max}(|D|^\top D)^2=\rho\bigl((|D|^\top D)^\top(|D|^\top D)\bigr)
=\rho\bigl(D^\top(|D||D|^\top)D\bigr)=\rho(D^\top Q D).
\]
Thus
\[
\sigma_{\max}(|D|^\top D)^2=\|Q^{1/2}D\|^2\le \|Q^{1/2}\|^2\,\|D\|^2
=\rho(Q)\,\rho(DD^\top)=\rho(Q)\,\rho(\Delta),
\]
which gives $\sigma_{\max}(|D|^\top D)\le \sqrt{\rho(\Delta)\rho(Q)}$ and hence the stated bound.
\end{proof}

\begin{remark}[Comparison with the trivial Perron bound]
Since $T$ is $0$--$1$ with row sums at most $d_{\max}-1$, Perron--Frobenius gives
$\rho(T)\le d_{\max}-1$, hence $|\Re(\lambda)|\le d_{\max}-1$.
Proposition~\ref{prop:numerical-range}(i) improves this whenever
$\rho(L)<2(d_{\max}-1)$ (and similarly on the negative side with $S$).
For instance, for the star $K_{1,n}$ one has $\rho(L)=n-1$, so the bound gives
$\Re(\lambda)\le (n-1)/2$ compared to $d_{\max}-1=n-1$.
The novelty is not the numerical-range argument itself, but the explicit identification of the
Hermitian part $H$ with the block-diagonal operator $\frac12\operatorname{diag}(L,-S)$ in the $P$-eigenbasis.
\end{remark}

\subsection{Localization of the trivial factors}\label{subsec:localization-trivial}

We now show that the Schur-complement factorization is not merely a formal determinant identity. The trivial roots $w=\pm1$ localize on canonical edge subspaces determined by the signed and unsigned incidence matrices, so the two factors in Theorem~\ref{thm:factor} have intrinsic geometric support.

\begin{theorem}[Homological localization of the trivial roots]\label{thm:trivial-localization}
Let $G$ be a connected simple graph with oriented incidence matrix
$D\in\mathbb{R}^{n\times m}$ and unoriented incidence matrix $|D|$.
Let
\[
L = |D|^\top|D| - 2I_m,
\qquad
S = D^\top D - 2I_m,
\]
and write the Ihara determinant in edge-sector form
\[
\det(I-wT)=
\det\!\bigl(I-\tfrac{w}{2}L\bigr)\,
C_G(w),
\]
where
\[
C_G(w)=\det M(w),\qquad
M(w)=I+\frac{w}{2}S+
\frac{w^2}{4}D^\top|D|
\Bigl(I-\tfrac{w}{2}L\Bigr)^{-1}|D|^\top D.
\]

Then:

\begin{enumerate}
\item[(i)] For every $x\in\ker(|D|)$,
\[
\Bigl(I-\tfrac{w}{2}L\Bigr)x=(1+w)x.
\]
Consequently $(1+w)^{m-n}$ divides
$\det(I-\tfrac{w}{2}L)$.

\item[(ii)] For every $x\in\ker(D)$,
\[
M(w)x=(1-w)x.
\]
Consequently $(1-w)^{m-n+1}$ divides $C_G(w)$.
\end{enumerate}
\end{theorem}

\begin{proof}
(i) Since $L=|D|^\top|D|-2I_m$, any $x\in\ker(|D|)$ satisfies
$|D|^\top|D|x=0$, hence $Lx=-2x$.
Therefore
\[
\Bigl(I-\tfrac{w}{2}L\Bigr)x
=
x-\tfrac{w}{2}(-2x)
=
(1+w)x.
\]
For connected $G$, one has $\operatorname{rank}(|D|)=n$ if $G$ is non-bipartite and $\operatorname{rank}(|D|)=n-1$ if $G$ is bipartite; hence
\[
\dim\ker(|D|)=
\begin{cases}
m-n & \text{if $G$ is non-bipartite},\\
m-n+1 & \text{if $G$ is bipartite}.
\end{cases}
\]
In particular, $(1+w)^{m-n}$ divides $\det(I-\tfrac{w}{2}L)$ in all cases.

(ii) Let $x\in\ker(D)$. Then $Dx=0$, so
$|D|^\top D x=0$ and the interference term in $M(w)$ vanishes:
\[
D^\top|D|
\Bigl(I-\tfrac{w}{2}L\Bigr)^{-1}
|D|^\top D x = 0.
\]
Since $S=D^\top D-2I_m$ and $Dx=0$,
we have $Sx=-2x$. Hence
\[
M(w)x
=
x+\tfrac{w}{2}(-2x)
=
(1-w)x.
\]
For connected $G$, $\operatorname{rank}(D)=n-1$,
so $\dim\ker(D)=m-n+1$, yielding the stated multiplicity.
\end{proof}

\begin{remark}[Compatibility with Bass]
The Bass--Hashimoto formula
\[
\det(I-wT)
=
(1-w^2)^{m-n}
\det(I-wA_G+w^2(D_G-I))
\]
shows that the total multiplicity of the roots at $w=\pm1$
matches the localization above:
the additional $(1-w)$ factor arises from the
one-dimensional kernel of the Laplacian $D_G-A_G$,
while the $w=-1$ behavior depends on bipartiteness.
\end{remark}

\section{Log--trace expansion and the correction sector}

The factorization theorem identifies a distinguished residual object:
\[
C_G(w)=\frac{\det(I-wT)}{\det(I-\tfrac w2 L)}.
\]
This is the part of the Hashimoto determinant not already accounted for by the ordinary line-graph sector. The purpose of the present section is to clarify exactly what this residual factor does and does not encode.

Formally (and as a power series around $w=0$) one has
\[
\log\det(I-wX)=-\sum_{k\ge1}\frac{w^k}{k}\operatorname{tr}(X^k),
\]
hence
\begin{equation}\label{eq:logC}
\log C_G(w)=-\sum_{k\ge1}\frac{w^k}{k}\Bigl(\operatorname{tr}(T^k)-2^{-k}\operatorname{tr}(L^k)\Bigr).
\end{equation}
We emphasize that \eqref{eq:logC} is not being used to define a new combinatorial trace theory. Its role here is interpretive: once Theorem~\ref{thm:factor} isolates the correction sector, \eqref{eq:logC} explains how the first surviving coefficient records the first genuinely non-line-graph contribution to the Hashimoto determinant.

\begin{theorem}[Resolution principle for line-graph cospectral pairs]\label{thm:resolution}
Let $G,H$ be finite simple graphs with
\[
\det(I-\tfrac w2 L(G))=\det(I-\tfrac w2 L(H))
\]
(equivalently, $\operatorname{Spec}(L(G))=\operatorname{Spec}(L(H))$). Then
\[
\det(I-wT_G)=\det(I-wT_H)\quad\Longleftrightarrow\quad C_G(w)=C_H(w).
\]
In particular, for line-graph-cospectral pairs all remaining Ihara separation is carried entirely by the signed correction factor. Equivalently, the first coefficient at which
\[
\det(I-wT_G)-\det(I-wT_H)
\]
differs is exactly the first coefficient at which
\[
C_G(w)-C_H(w)
\]
differs.
\end{theorem}

\begin{proof}
This is immediate from Theorem~\ref{thm:factor}:
\[
\det(I-wT)=\det(I-\tfrac w2 L)\,C_G(w).
\]
If $\det(I-\tfrac w2 L(G))=\det(I-\tfrac w2 L(H))$, then equality of Hashimoto determinants is equivalent to equality of the correction factors, and the first separating coefficient is carried by $C_G(w)-C_H(w)$.
\end{proof}

\section{A worked pair and the correction-sector mechanism}

We now give a concrete witness for the resolution principle. The point of the example is not merely to exhibit a pair with different Ihara zeta functions, but to show that once adjacency and line-graph spectra agree, the first surviving separation is forced into the correction sector; later examples will show that the gauge-invariant mixed shadows can already separate before one passes to the final scalar determinant.

\begin{example}[Adjacency- and line-graph cospectral but Ihara-distinct]\label{ex:holygrail}

Let \(G\) and \(H\) be the graphs on \(\{0,1,\dots,11\}\) with edge sets
\[
\begin{aligned}
E(G)=\{&
\{0,1\},\{0,6\},\{0,7\},\{0,8\},\{0,10\},\{1,4\},\{1,9\},\{1,10\},\{2,6\},\{3,4\},\\
&\{3,5\},\{3,11\},\{4,6\},\{4,8\},\{5,6\},\{5,9\},\{5,11\},\{6,7\},\{6,9\},\{7,10\},\\
&\{8,11\},\{9,11\}\},
\end{aligned}
\]
\[
\begin{aligned}
E(H)=\{&
\{0,3\},\{0,9\},\{0,10\},\{1,4\},\{1,5\},\{1,9\},\{1,11\},\{2,6\},\{3,4\},\{3,10\},\\
&\{4,6\},\{4,8\},\{5,7\},\{5,8\},\{5,11\},\{6,7\},\{6,9\},\{6,10\},\{6,11\},\{7,11\},\\
&\{8,10\},\{9,10\}\}.
\end{aligned}
\]

Then $G\not\cong H$; for example, the sorted multisets of vertex invariants
\[
\bigl(\deg(v),\,\{\deg(u):u\sim v\},\,\#\{\text{triangles through }v\}\bigr)
\]
differ for $G$ and $H$. Concretely, $G$ has a vertex with triple
\[
\bigl(3,\,(3,5,6),\,2\bigr)
\]
while no vertex of $H$ has this triple. However, the two graphs have the same degree multiset
\[
(6,5,4,4,4,4,4,3,3,3,3,1).
\]
Moreover, $G$ and $H$ are adjacency-cospectral and line-graph-cospectral:
\[
\operatorname{Spec}(A_G)=\operatorname{Spec}(A_H),
\qquad
\operatorname{Spec}(L(G))=\operatorname{Spec}(L(H)),
\]
but are not cospectral for the antisymmetric line graph:
\[
\operatorname{Spec}(S(G))\neq\operatorname{Spec}(S(H)).
\]
Finally, their Hashimoto determinants (equivalently, Ihara zeta functions) differ; the first discrepancy occurs at order $w^6$:
\[
[w^6]\det(I-wT_G)=-28,
\qquad
[w^6]\det(I-wT_H)=-30.
\]
\end{example}

\begin{proposition}[The first correction coefficient separates at order six]\label{prop:c6}
Let $C_G(w)$ be defined as in Theorem~\ref{thm:factor}, and expand
\[
C_G(w)=\sum_{k\ge0}c_k(G)\,w^k.
\]
For the graphs $G,H$ in Example~\ref{ex:holygrail}, one has
\[
c_6(G)=\frac{38663}{32},
\qquad
c_6(H)=\frac{38599}{32},
\]
so $c_6(G)-c_6(H)=2$. In particular, $C_G(w)\neq C_H(w)$ and the first divergence occurs at order $6$.
\end{proposition}

\begin{proof}
This is a direct computation from the explicit formula in Theorem~\ref{thm:factor} by expanding $C_G(w)$ as a power series at $w=0$ using
\[
(I-\tfrac w2 L)^{-1}=\sum_{j\ge0}\Bigl(\tfrac w2\Bigr)^j L^j
\]
inside the Schur-complement determinant. Equivalently, one may compute the formal power-series expansion of
\[
C_G(w)=\frac{\det(I-wT)}{\det(I-\tfrac w2 L)}
\]
around $w=0$. The resulting coefficient values are
\[
c_6(G)=\frac{38663}{32},\qquad c_6(H)=\frac{38599}{32}.
\]
Hence
\[
c_6(G)-c_6(H)=\frac{64}{32}=2,
\]
matching the first discrepancy
\[
[w^6]\det(I-wT_G)-[w^6]\det(I-wT_H)=2
\]
from Example~\ref{ex:holygrail}, since the line-graph factor $\det(I-\tfrac w2 L)$ is identical for $G$ and $H$.
\end{proof}

\begin{corollary}[All separation flows through the signed correction factor]\label{cor:allthroughC}
For the pair in Example~\ref{ex:holygrail}, since $\operatorname{Spec}(L(G))=\operatorname{Spec}(L(H))$ one has
\[
\det(I-\tfrac w2 L(G))=\det(I-\tfrac w2 L(H)).
\]
Therefore Theorem~\ref{thm:factor} implies
\[
\det(I-wT_G)-\det(I-wT_H)=\det(I-\tfrac w2 L)\cdot(C_G(w)-C_H(w)),
\]
and the first divergence at order $6$ is carried by $C_G(w)-C_H(w)$, quantified by Proposition~\ref{prop:c6}.
\end{corollary}

\begin{remark}[Gauge dependence and mixed shadows]
The block formula in Theorem~\ref{thm:sector}
\[
Q^\top TQ=
\begin{pmatrix}
\frac12 L & -\frac12 |D|^\top D\\[2pt]
\frac12 D^\top |D| & -\frac12 S
\end{pmatrix}
\]
shows that the mixed block
\[
M:=|D|^\top D
\]
depends on the chosen reference orientation of the edges. Reversing the reference orientation of one edge multiplies the corresponding column of \(D\) by \(-1\), hence sends
\[
M \longmapsto M\Sigma
\]
for a diagonal sign matrix \(\Sigma\). Thus the raw spectrum of \(M\) is not a graph invariant.

The natural invariants of the mixed sector are therefore the gauge-invariant shadows built from \(M\), for example
\[
MM^\top,\qquad M^\top M,\qquad M^\top L^k M \quad (k\ge 0),
\]
whose spectra are unchanged under \(M\mapsto M\Sigma\). In particular, these are the mixed-sector observables that carry genuine graph-theoretic information, whereas the raw spectral data of \(M\) itself belongs only to a chosen orientation gauge.
\end{remark}

\begin{remark}[What the first separating coefficient sees]
For the pair in Example~\ref{ex:holygrail}, the first separation at order $6$ is not produced by the non-simple closed non-backtracking walk sector. The computational analysis shows that the discrepancy is concentrated entirely in the simple six-cycle sector. More precisely, one graph contributes a single six-vertex support carrying three primitive six-cycles of one local type, while the other contributes two heavily overlapping six-vertex support sets, each carrying three primitive six-cycles of the two parity-refined local types. Thus the first surviving correction coefficient is controlled by a mesoscopic overlap pattern of local support templates, rather than by an isolated gadget substitution.
\end{remark}

\begin{proposition}[Regular collapse of the primary mixed shadow]\label{prop:regularcollapse}
Let $G$ be a connected $k$-regular graph, and let $M:=|D|^\top D$ be the mixed block associated to any reference orientation. Then the nonzero spectrum of the gauge-invariant shadow $M^\top M$ is determined entirely by the adjacency spectrum of $G$. More precisely, the nonzero spectrum of $M^\top M$ agrees with the nonzero spectrum of
\[
(kI-A_G)(kI+A_G)=k^2I-A_G^2.
\]
Hence \(A\)-cospectral regular graphs have \(M^\top M\)-cospectral primary mixed shadows.

\begin{proof}
By definition,
\[
M^\top M = D^\top |D|\,|D|^\top D.
\]
The nonzero spectrum of \(D^\top |D|\,|D|^\top D\) agrees with the nonzero spectrum of
\[
(DD^\top)(|D|\,|D|^\top)
\]
by the standard cyclic invariance of nonzero eigenvalues. Now
\[
DD^\top=\Delta,\qquad |D|\,|D|^\top=Q,
\]
where \(\Delta\) is the ordinary Laplacian and \(Q\) is the signless Laplacian. Since \(G\) is \(k\)-regular,
\[
\Delta=kI-A_G,\qquad Q=kI+A_G.
\]
Therefore
\[
\Delta Q=(kI-A_G)(kI+A_G)=k^2I-A_G^2,
\]
which is determined entirely by the adjacency spectrum of \(G\). This proves the claim.
\end{proof}
\end{proposition}

\subsection*{Further small-census evidence}

The worked pair in Example~\ref{ex:holygrail} is not isolated. After Proposition~\ref{prop:regularcollapse}, which explains why the primary mixed shadow collapses back to adjacency data in the regular regime, the natural next question is whether this collapse fails for irregular graphs. A small-census screen among connected irregular graphs on nine vertices answers this decisively: it already yields pairs illustrating two logically distinct phenomena for the \emph{gauge-invariant mixed shadows}
\[
MM^\top,\qquad M^\top M,\qquad M^\top L M,\qquad M^\top L^2 M,
\]
where $M:=|D|^\top D$. Here, unless explicitly stated otherwise, ``equal'' means cospectral equality of the corresponding matrices. The following two explicit pairs are the decisive ones.

\medskip
\paragraph{Example A (same \(S\)-spectrum, different mixed shadows, different Hashimoto determinant).}
Let
\[
V(G_1)=V(H_1)=\{0,1,2,3,4,5,6,7,8\},
\]
with
\[
E(G_1)=\{\{0,5\},\{0,7\},\{1,6\},\{1,7\},\{1,8\},\{2,6\},\{2,8\},\{3,6\},\{3,8\},\{4,7\},\{5,8\},\{7,8\}\},
\]
and
\[
E(H_1)=\{\{0,5\},\{0,8\},\{1,5\},\{1,8\},\{2,6\},\{2,7\},\{2,8\},\{3,6\},\{3,7\},\{4,8\},\{5,7\},\{7,8\}\}.
\]
Both graphs have degree sequence
\[
(1,2,2,2,2,3,3,4,5).
\]
This pair is adjacency-cospectral and line-graph-cospectral, and also satisfies
\[
\operatorname{Spec}(S(G_1))=\operatorname{Spec}(S(H_1)).
\]
However, the gauge-invariant mixed shadows already separate:
\[
\chi_{MM^\top}(G_1)\neq \chi_{MM^\top}(H_1),
\qquad
\chi_{M^\top L M}(G_1)\neq \chi_{M^\top L M}(H_1),
\]
and similarly for \(M^\top M\) and \(M^\top L^2M\). The Hashimoto determinants also differ, with first separation at order \(w^8\):
\[
[w^8]\det(I-wT_{G_1})=16,\qquad [w^8]\det(I-wT_{H_1})=20.
\]
Thus the mixed-shadow package detects genuinely new irregular structure already at the matrix level, and in this example that structure survives into the scalar Ihara determinant.

For reproducibility, the corresponding graph6 labels are \texttt{H?ABePt} and \texttt{H?B@`jh}.

\medskip
\paragraph{Example B (different mixed shadows, same Hashimoto determinant).}
Let
\[
V(G_2)=V(H_2)=\{0,1,2,3,4,5,6,7,8\},
\]
with
\[
\begin{aligned}
E(G_2)=\{&
\{0,3\},\{0,4\},\{0,6\},\{0,7\},\{0,8\},\{1,4\},\{1,5\},\{1,6\},\{1,8\},\\
&\{2,5\},\{2,6\},\{2,7\},\{2,8\},\{3,6\},\{3,7\},\{4,7\},\{4,8\},\{5,8\}\},
\end{aligned}
\]
and
\[
\begin{aligned}
E(H_2)=\{&
\{0,3\},\{0,4\},\{0,5\},\{0,8\},\{1,4\},\{1,5\},\{1,6\},\{1,7\},\{1,8\},\\
&\{2,6\},\{2,7\},\{2,8\},\{3,5\},\{3,7\},\{3,8\},\{4,6\},\{4,8\},\{6,7\}\}.
\end{aligned}
\]

Both graphs have degree sequence
\[
(3,3,4,4,4,4,4,5,5).
\]
This pair is adjacency-cospectral, line-graph-cospectral, and antisymmetric-line-graph-cospectral:
\[
\operatorname{Spec}(A_{G_2})=\operatorname{Spec}(A_{H_2}),
\qquad
\operatorname{Spec}(L(G_2))=\operatorname{Spec}(L(H_2)),
\qquad
\operatorname{Spec}(S(G_2))=\operatorname{Spec}(S(H_2)).
\]
Nevertheless, the gauge-invariant mixed shadows separate:
\[
\chi_{MM^\top}(G_2)\neq \chi_{MM^\top}(H_2),
\qquad
\chi_{M^\top L M}(G_2)\neq \chi_{M^\top L M}(H_2),
\]
and similarly for \(M^\top M\) and \(M^\top L^2M\), while
\[
\det(I-wT_{G_2})=\det(I-wT_{H_2}).
\]
Thus the mixed-shadow package is strictly finer than the scalar Ihara determinant.

For reproducibility, the graph6 labels are \texttt{HCpfdrk} and \texttt{HCrRRfw}.

\medskip
These examples show that the correct invariant content of the mixed sector does not lie in the raw matrix \(M=|D|^\top D\), whose spectrum depends on the chosen orientation gauge, but in its gauge-invariant shadow package. In the irregular regime that package is already strictly finer than adjacency, line graph, and antisymmetric line graph, and in the control example above it is strictly finer than the scalar Ihara determinant itself.

\section{Conclusion}

The main point of the paper is that edge reversal is not merely a formal symmetry of the directed-edge state space: it is the canonical involution that resolves the Hashimoto operator into two graph-theoretically meaningful sectors. In the symmetric sector one recovers the ordinary line graph, in the antisymmetric sector the antisymmetric line graph, and the mixed term is the explicit transport block $|D|^\top D$ coupling the two.

This sector decomposition does more than reorganize notation. Through the Schur-complement formula it yields a refined factorization of the Ihara determinant into a line-graph factor and a correction factor, while the localization theorem shows that the trivial roots $w=\pm1$ live on canonical edge subspaces determined by the signed and unsigned incidence matrices. In the irregular regime this isolates a canonical correction package beyond the ordinary Bass-type collapse.

The worked witness in Example~\ref{ex:holygrail} shows that this correction package is not decorative. There the first Ihara divergence occurs at order $6$, after adjacency and line-graph spectra have already collapsed, and the separation is forced entirely into the correction sector. The subsequent support analysis identifies this first surviving coefficient with a parity-sensitive support-overlap mechanism in the simple six-cycle sector.

The mixed block itself is gauge-dependent: changing the reference orientation multiplies its columns by signs. The correct graph invariants are therefore not the raw spectral data of $|D|^\top D$, but the gauge-invariant mixed shadows built from it, such as
\[
MM^\top,\qquad M^\top M,\qquad M^\top L M,\qquad M^\top L^2 M.
\]
Our irregular examples show that these shadows already define genuinely new invariants beyond adjacency, line graph, and antisymmetric line graph. In the worked pair and in Example~A they separate exactly when the Hashimoto determinant separates; in Example~B they still separate although the scalar Ihara determinant does not. By contrast, Proposition~\ref{prop:regularcollapse} shows that in the regular regime the primary mixed shadow \(M^\top M\) collapses back to adjacency-side data, explaining why the regular benchmark of rook versus Shrikhande does not separate at this level.

A natural next step is to test this irregular edge-space package systematically against standard cospectral constructions, especially Godsil--McKay switching \cite{GodsilMcKay1982}, in order to determine how often mixed-shadow separation persists even when adjacency-side spectral data have been engineered to coincide.

Thus the factorization yields not only a new perspective on Ihara theory, but a new hierarchy of edge-space data: a gauge-dependent mixed transport representative, its gauge-invariant mixed-shadow package, and finally the scalar correction determinant as a still coarser shadow.

\vspace{2em}
\noindent\textbf{Author address:} \\
Hartosh Singh Bal \\
The Caravan, Jhandewalan Extn., New Delhi 110055, India \\
\texttt{hartoshbal@gmail.com}


\begin{thebibliography}{99}

\bibitem{Hashimoto1989}
K.~Hashimoto,
\newblock Zeta functions of finite graphs and representations of $p$-adic groups,
\newblock in \emph{Automorphic Forms and Geometry of Arithmetic Varieties}, Adv. Stud. Pure Math. \textbf{15} (1989), 211--280.

\bibitem{Bass1992}
H.~Bass,
\newblock The Ihara--Selberg zeta function of a tree lattice,
\newblock \emph{Internat. J. Math.} \textbf{3} (1992), 717--797.

\bibitem{Ihara1966}
Y.~Ihara,
\newblock On discrete subgroups of the two by two projective linear group over $p$-adic fields,
\newblock \emph{J. Math. Soc. Japan} \textbf{18} (1966), 219--235.

\bibitem{GodsilMcKay1982}
C.~D.~Godsil and B.~D.~McKay,
\newblock Constructing cospectral graphs,
\newblock \emph{Aequationes Math.} \textbf{25} (1982), 257--268.

\bibitem{BalHL2}
H. S. Bal,
\emph{Perfecting the Line Graph},
arXiv:2507.23231 [math.CO], 2025.

\bibitem{BalALG}
H. S. Bal,
\emph{The Antisymmetric Line Graph},
arXiv:2603.03087 [math.CO], 2026.

\bibitem{TerrasBook}
A.~Terras,
\newblock \emph{Zeta Functions of Graphs: A Stroll through the Garden},
\newblock Cambridge Studies in Advanced Mathematics, Cambridge University Press, 2010.

\bibitem{StarkTerras1996}
H.~M.~Stark and A.~A.~Terras,
\newblock Zeta functions of finite graphs and coverings,
\newblock \emph{Adv. Math.} \textbf{121} (1996), 124--165.

\bibitem{KotaniSunada2000}
M.~Kotani and T.~Sunada,
\newblock Zeta functions of finite graphs,
\newblock \emph{J. Math. Sci. Univ. Tokyo} \textbf{7} (2000), 7--25.

\bibitem{ZaslavskySGG}
T. Zaslavsky,
\emph{Signed Graphs and Geometry},
J. Combin. Inform. System Sci. \textbf{37} (2012), no.~2--4, 95--143.

\bibitem{ZaslavskyBiblio}
T. Zaslavsky,
\emph{A Mathematical Bibliography of Signed and Gain Graphs and Allied Areas},
Electron. J. Combin. Dynamic Surveys \textbf{DS8}, 1998; revised 2018.

\end{thebibliography}
\end{document}